\documentclass{article}
\usepackage{graphicx,amsmath,amssymb,amsthm,xcolor,hyperref} 

\everymath{\displaystyle}

\newtheorem{thm}{Theorem}[section]
\newtheorem{lemma}[thm]{Lemma}
\newtheorem{cor}[thm]{Corollary}
\renewcommand{\arraystretch}{1.5}

\title{Analytic summation of series involving higher-order derivatives of Chebyshev polynomials of the second kind and their applications to convolved linear recurrent sequences}
\author{Khamitov V., Dmytryshyn D.,  Gray D., and Stokolos A.}
\date{March 2026}

\begin{document}

\maketitle

\abstract{This paper considers functional series whose terms are higher-order derivatives of Chebyshev
polynomials of the second kind, where the degree of the polynomial is related to the order of the
derivative. Analytic summation is used to determine the rational functions to which these series
converge. These functions are expressed in terms of Chebyshev polynomials evaluated at a
specific argument. Connections are established between derivatives of Chebyshev polynomials
of the second kind and special numerical sequences generated by linear recurrence relations.
New closed-form formulas are obtained for the sums of the series at various values of the
argument. As consequences, combinatorial identities are derived for the Fibonacci, Lucas, and
Pell numbers, for sections of the Fibonacci sequence, and for their convolutions. By means of
analytic continuation, sums of formally divergent series are obtained, which in special cases
correspond to the classical Euler formulas.

}

\section{Introduction and Problem Statement}

In \cite{ref1, ref2, ref3}, it was noted that identifying internal connections among Chebyshev polynomials of the second kind, their derivatives, and analytic functions is an important and interesting problem and, moreover, that the results in this direction are clearly insufficient. This paper considers specific series associated with derivatives of Chebyshev polynomials of the second kind, in which the degree of the polynomial is related to the order of the derivative, and which converge to rational functions. These functions are expressed through Chebyshev polynomials with a composite argument. Consequently, various identities are obtained for series related to the Fibonacci, Lucas, and Pell numbers, to the convolved Fibonacci and Pell numbers, and to the $k$-section numbers of the Fibonacci sequence. Further consequences concern the summation of formally divergent series by analytic continuation to the boundary of the domain of convergence of these series.

For every integer $N \geq 0$, the family of Chebyshev polynomials of the second kind $U_N(x)$ is defined as follows; see, for example, \cite{ref1}:
\begin{equation}\label{eq1}
U_N(z) = \sum_{j = 0}^{\left[\frac{N}{2}\right]}(-1)^j{N - j \choose j}(2z)^{N - 2j},
\end{equation}
where
$$
\left[\frac{N}{2}\right] = \begin{cases}
\frac{N}{2}, & N \mbox{ is even} \\
\\
\frac{N-1}{2}, & N \mbox{ is odd}
\end{cases}, \qquad
{N - j \choose j} = \frac{(N - j)!}{j!(N - 2j)!}.
$$
For $N < 0$, the family of Chebyshev polynomials of the second kind can be defined by the formula $U_{-N}(z) = -U_{N - 2}(z)$.

Formula \eqref{eq1} makes it possible to define the polynomials $U_N(z)$ for all complex $z$, and also yields explicit formulas for derivatives of order $s$ ($0 \leq s \leq N$) of the Chebyshev polynomials of the second kind:
$$
U_N^{(s)}(z) = s!\sum_{j = 0}^{\left[\frac{N - s}{2}\right]}(-1)^j{N - j \choose j}{N - 2j \choose s}2^{N - 2j}z^{N - 2j - s}.
$$

Formulas connecting derivatives of Chebyshev polynomials of the second kind to the Chebyshev polynomials of the second kind themselves can be found in \cite{ref1,ref2,ref3,ref4}.

The main objective of the present work is: 
\begin{itemize}
\item to derive formulas for the sums of the series
$$
\sum_{\sigma = 0}^{\infty}\frac{z^\sigma U_{N + \sigma - 1}^{(\sigma)}(z)}{2^\sigma \sigma!}, \quad |z| < 1 \quad \mbox{ and } \quad \sum_{\sigma = 0}^{\infty}\frac{z^{-\sigma} U_{N + \sigma - 1}^{(\sigma)}(z)}{2^\sigma \sigma!}, \quad |z| > 1;
$$
\item to obtain special formulas for certain linear recurrent sequences; 
\item and to derive summation formulas for certain divergent series by analytic continuation of the resulting rational functions.
\end{itemize}

The relevance of the present work is due to the growing interest in the search for analytic representations of various types of sums involving orthogonal polynomials, in particular Chebyshev polynomials of the second kind. These polynomials are widely used in approximation theory, numerical methods, and cryptography \cite{ref5}. However, whereas the properties of the polynomials themselves have been studied in detail, series involving their higher-order derivatives remain less investigated \cite{ref6}. One of the motivations for writing this article is the possibility of extending properties of derivatives $U_{N + s - 1}^{(s)}(z)$ to number theory. Establishing exact values for the sums of such series not only makes it possible to discover new properties of convolved recurrent sequences, but also provides tools for summing certain divergent series in the sense of analytic continuation.

\section{Preliminary Results}

Let us determine the connection between derivatives of Chebyshev polynomials of the second kind $U_{N + s - 1}^{(s)}(z)$ and the Fibonacci, Lucas, and Pell numbers, the $k$-section of the Fibonacci sequence, the convolved Fibonacci numbers, the convolved Pell numbers, and the convolved $k$-section of the Fibonacci sequence \cite{ref7, ref8}. We present some facts from the theory of Fibonacci numbers.

The Fibonacci sequence $\{F_N\}_{N = 1}^{\infty}$ is defined by the recurrence relations
$$
x_{N + 2} = x_{N + 1} + x_N, \quad N = 1,2,\ldots,
$$
and the initial conditions $x_1 = 1$, $x_2 = 1$. In turn, the Lucas sequence $\{L_N\}_{N = 1}^{\infty}$ is defined by the same relations and the initial conditions $x_1 = 1$, $x_2 = 3$. The Pell sequence $\{P_N\}_{N = 1}^{\infty}$ is defined by a different linear recurrence relation
$$
x_{N + 2} = 2x_{N + 1} + x_N, \quad N = 1,2,\ldots,
$$
and the initial conditions $x_1 = 1$, $x_2 = 2$.

Denote $\Phi_{N,k} = F_{Nk}/F_k$, $N = 1,2, \ldots$, $k = 1,2, \ldots$. For each $k = 1,2,\ldots,$ the sequence $\{\Phi_{N,k}\}_{N = 1}^{\infty}$ is defined by the recurrence relation
$$
x_{N + 2} = L_kx_{N+1} - (-1)^kx_N, \qquad N = 1,2,\ldots,
$$
and the initial conditions $x_1 = 1$, $x_2 = L_k$ \cite{ref8,ref9}, or directly in terms of the Lucas numbers \cite{ref10}
$$
\Phi_{N,k} = \sum_{j = 0}^{\left[\frac{N-1}{2}\right]}(-1)^{(k-1)j}{N - 1 - j \choose j}(L_k)^{N - 1 - 2j}.
$$
We shall call this sequence $\{\Phi_{N,k}\}_{N = 1}^{\infty}$ the $k$-section of the Fibonacci sequence.

The following generalization of sequences is associated with convolution operations. The convolved Fibonacci number is defined as follows \cite{ref7, ref10}:
\begin{equation}\label{eq2}
F_N^{(0)} = F_N, \quad F_N^{(s)} = \sum_{j = 0}^{N - 1}F_{j + 1}F_{N - j}^{(s - 1)}, \quad s = 1,2,3,\ldots.
\end{equation}

It follows from formula \eqref{eq2} that 
\begin{eqnarray*}
F_1^{(s)} &=& 1, \\
F_2^{(s)} &=& s + 1, \\
F_3^{(s)} &=& \frac{1}{2}(s + 1)(s + 4), \\
F_4^{(s)} &=& \frac{1}{6}(s + 1)(s + 2)(s + 9), \\
F_5^{(s)} &=& \frac{1}{24}(s + 1)(s + 2)(s + 4)(s + 15),
\end{eqnarray*}
and so on.

A formula relating convolved Fibonacci numbers is presented in \cite[formula (18)]{ref7}.

The convolved Pell number is defined analogously to \eqref{eq2}:
$$
P_N^{(0)} = P_N, \quad P_N^{(s)} = \sum_{j = 0}^{N - 1}P_{j + 1}P_{N - j}^{(s - 1)}, \quad s = 1,2,3,\ldots.
$$
Note that
\begin{eqnarray*}
P_1^{(s)} &=& 1, \\
P_2^{(s)} &=& 2(s + 1), \\
P_3^{(s)} &=& (s + 1)(2s + 5), \\
P_4^{(s)} &=& \frac{2}{3}(s + 1)(s + 2)(2s + 9), \\
P_5^{(s)} &=& \frac{1}{6}(s + 1)(s + 2)(4s^2 + 40s + 87).
\end{eqnarray*}

For the $k$-sections of the Fibonacci sequence, one can also define convolved numbers:
$$
\Phi_{N,k}^{(0)} = \Phi_{N,k}, \quad \Phi_{N,k}^{(s)} = \sum_{j = 0}^{N - 1} \Phi_{j + 1,k}\Phi_{N - j,k}^{(s - 1)}, \quad s = 1,2,3,\ldots. 
$$

Since $\Phi_{N,1} = F_N$, we have $\Phi_{N,1}^{(s)} = F_{N}^{(s)}$.

We also note that
\begin{eqnarray*}
\Phi_{1,k}^{(s)} &=& 1, \\
\Phi_{2,k}^{(s)} &=& (s + 1)L_k, \\
\Phi_{3,k}^{(s)} &=& \frac{1}{2}(s + 1)(s + 2)(L_k)^2 + (s + 1)(-1)^k, \\
\Phi_{4,k}^{(s)} &=& \frac{1}{2}(s + 1)(s + 2)(s + 3)(L_k)^3 - (s + 1)(s + 2)(-1)^kL_k. \\
\end{eqnarray*}

There exist formulas relating the $k$-sections of the Fibonacci sequence to the Fibonacci numbers and to the Lucas numbers \cite[formulas (21), (22)]{ref7}.

Formulas relating the derivatives of the Chebyshev polynomials of the second kind to Fibonacci numbers, Lucas numbers, convolved Fibonacci numbers, and the $k$-sections of the Fibonacci sequence are given in \cite{ref7,ref8}:
$$
F_N = (-i)^{N - 1}U_{N - 1}\left(\frac{i}{2}\right),
$$
$$
\Phi_{N,k} = \frac{F_{Nk}}{F_k} = \begin{cases}
(-i)^{N - 1}U_{N - 1}\left(\frac{i}{2}L_k\right), & k \mbox{ is odd} \\
U_{N - 1}\left(\frac{1}{2}L_k\right), & k \mbox{ is even}
\end{cases},
$$
\vspace{2mm}
\begin{equation}\label{eq3}
F_N^{(s)} = \frac{(-i)^{N - 1}}{2^s s!}U_{N + s - 1}^{(s)}\left(\frac{i}{2}\right),
\end{equation}
\vspace{2mm}
\begin{equation}\label{eq4}
\Phi_{N,k}^{(s)} = \frac{1}{2^s s!}\begin{cases}
(-i)^{N - 1}U_{N + s - 1}^{(s)}\left(\frac{i}{2}L_k\right), & k \mbox{ is odd} \\
U_{N + s - 1}^{(s)}\left(\frac{1}{2}L_k\right), & k \mbox{ is even}
\end{cases}
\end{equation}
\vspace{2mm}
\begin{lemma}\label{lem1}
For convolved Pell numbers, the following relation holds
\begin{equation}\label{eq5}
P_N^{(s)} = \frac{(-i)^{N - 1}}{2^s s!}U_{N + s - 1}^{(s)}(i).
\end{equation}
\end{lemma}

\begin{proof}
The function $p(z) = \frac{z}{1 - 2z - z^2}$ is the generating function for the Pell sequence \cite{ref8}, i.e.
$$
\frac{z}{1 - 2z - z^2} = \sum_{j = 1}^{\infty}P_j z^j.
$$
For the convolved Pell numbers $P_N^{(s)}$, the generating function is defined as
$$
p^{(s)}(z) = \frac{z}{(1 - 2z - z^2)^{s + 1}}.
$$
Indeed, using \cite[p. 216]{ref6}, we obtain
$$
\frac{z}{1 - 2z - z^2} \cdot \frac{1}{1 - 2z - z^2} = \sum_{j = 1}^{\infty}\left(\sum_{\ell = 0}^{j - 1}P_{\ell + 1}P_{j - \ell}\right)z^j = \sum_{j = 1}^{\infty}P_j^{(1)}z^j,
$$
$$
\frac{z}{(1 - 2z - z^2)^s} \cdot \frac{1}{1 - 2z - z^2} = \sum_{j = 1}^{\infty}\left(\sum_{\ell = 0}^{j - 1}P_{\ell + 1}P_{j - \ell}^{(s - 1)}\right)z^j = \sum_{j = 1}^{\infty}P_j^{(s)}z^j.
$$
The function $g(z,t) = \frac{1}{1 - 2tz + z^2}$ is the generating function for the Chebyshev polynomials of the second kind $U_N(t)$ \cite{ref5}, so that
$$
\frac{1}{1 - 2tz + z^2} = \sum_{j = 0}^{\infty}U_j(t)z^j.
$$
Let us determine the higher-order derivatives of $g(z,t)$.
\begin{eqnarray*}
\frac{\partial^s}{\partial t^s}g(z,t) &=& 2^s s!\frac{z^s}{(1 - 2tz + z^2)^{s + 1}} \\
&=& \sum_{j = 0}^{\infty}U_j^{(s)}(t)z^j \\
&=& \sum_{j = s}^{\infty}U_j^{(s)}(t)z^j \\
&=& \sum_{j = 1}^{\infty}U_{j + s - 1}^{(s)}(t)z^{j + s - 1} \\
& & (U_j^{(s)}(t) = 0, \quad j < s).
\end{eqnarray*}
Then,
$$
 \frac{z}{(1 - 2tz + z^2)^{s + 1}} = \frac{1}{2^s s!}\sum_{j = 1}^{\infty}U_{j + s - 1}^{(s)}(t) z^j.
$$
Therefore, $P_N^{(s)} = \frac{(-i)^{N - 1}}{2^s s!}U_{N + s - 1}^{(s)}(i)$.
\end{proof}

Next, series of the forms
\begin{itemize}
\item $\displaystyle \sum_{\sigma = 0}^{\infty} \frac{z^{\pm \sigma}U_{N + \sigma - 1}^{(\sigma)}(z)}{2^\sigma \sigma!}$,
\item $\displaystyle \sum_{\sigma = 0}^{\infty} (i)^{\sigma}2^{-\sigma}F_N^{\sigma}$,
\item $\displaystyle \sum_{\sigma = 0}^{\infty}\left(\frac{2}{L_k}\right)^\sigma \Phi_{N,k}^{(\sigma)}$,
\item and $\displaystyle \sum_{\sigma = 0}^{\infty}\left(\frac{-2i}{L_k}\right)^\sigma \Phi_{N,k}^{(\sigma)}$
\end{itemize}
will be studied. As consequences of the obtained formulas, formally divergent series will also be considered; for example,
$$
\sum_{\sigma = 0}^{\infty}(i)^\sigma P_N^{(\sigma)}, \qquad \sum_{\sigma = 0}^{\infty}(-1)^\sigma{N + 2\sigma \choose N - 1}.
$$

\section{Main Results}

Denote $A_{N,s,j} = {N - 1 + s - j \choose j}{N - 1 + s - 2j \choose s}2^{N - 1 - 2j}$. Then,
\begin{equation}\label{eq6}
\frac{1}{2^s s!}U_{N + s - 1}^{(s)}(z) = \sum_{j = 0}^{\left[\frac{N - 1}{2}\right]}(-1)^jA_{N,s,j}z^{N-1-2j}.
\end{equation}

\begin{lemma}\label{lem2}
For all $j = 0, 1, \ldots, \left[\frac{N-1}{2}\right]$ and $|z| < 1$, the representation
\begin{equation}\label{eq7}
\sum_{\sigma = 0}^{\infty}A_{N,\sigma,j}z^\sigma = {N - 1 - j \choose j}\frac{2^{N - 1 - 2j}}{(1 - z)^{N - j}}
\end{equation}
holds.
\end{lemma}

\begin{proof}
By Newton's binomial formula,
$$
\frac{1}{(1 - z)^{N - j}} = \sum_{\sigma = 0}^{\infty}{N - 1 - j + \sigma \choose \sigma}z^\sigma.
$$
Let us compare the coefficients of the series in \eqref{eq7}. On the left-hand-side the coefficients are equal to
$$
{N - 1 + \sigma - j \choose j}{N - 1 + \sigma - 2j \choose \sigma}2^{N - 1 - 2j},
$$
and on the right-hand-side are equal to
$$
{N - 1 - j \choose j}{N - 1 - j + \sigma \choose \sigma}2^{N - 1- 2j}.
$$
It is easy to verify that these coefficients coincide.
\end{proof}

\begin{lemma}\label{lem3}
The following relations hold:
\begin{equation}\label{eq8}
\sum_{\sigma = 0}^{\infty}\frac{z^\sigma U_{N+\sigma-1}^{(\sigma)}(z)}{2^\sigma \sigma!} = \sum_{j = 0}^{\left[\frac{N-1}{2}\right]}(-1)^j{N-1-j \choose j}\frac{(2z)^{N-1-2j}}{(1-z)^{N-j}}, \quad |z| < 1;
\end{equation}
\begin{eqnarray}\label{eq9}
\lefteqn{\sum_{\sigma = 0}^{\infty}\frac{z^{-\sigma} U_{N+\sigma-1}^{(\sigma)}(z)}{2^\sigma \sigma!}} \nonumber \\
&=& \sum_{j = 0}^{\left[\frac{N-1}{2}\right]}(-1)^j{N-1-j \choose j}\frac{2^{N-1-2j}z^{2N-1-3j}}{(z - 1)^{N-j}}, \quad |z| > 1.
\end{eqnarray}
\end{lemma}

\begin{proof}
It follows from \eqref{eq6} that $\displaystyle \frac{U_{N+s-1}^{(s)}(z)}{2^s s!} \sim  \frac{2^{N-1}}{(N-1)!}s^{N-1}z^{N-1}$ for $s \rightarrow \infty$. Consequently, the series $\displaystyle \sum_{\sigma = 0}^{\infty}\frac{z^\sigma U_{N+\sigma-1}^{(\sigma)}(z)}{2^\sigma \sigma!}$ converges for $|z| < 1$. Since
\begin{eqnarray*}
\sum_{\sigma = 0}^{\infty}\frac{z^\sigma U_{N+\sigma-1}^{(\sigma)}(z)}{2^\sigma \sigma!} &=& \sum_{\sigma = 0}^{\infty}z^\sigma\sum_{j = 0}^{\left[\frac{N-1}{2}\right]}(-1)^jA_{N,\sigma,j}z^{N-1-2j} \\
&=& \sum_{j = 0}^{\left[\frac{N-1}{2}\right]}(-1)^jz^{N-1-2j}\sum_{\sigma = 0}^{\infty}A_{N,\sigma,j}z^\sigma,
\end{eqnarray*}
taking Lemma~\ref{lem2} into account, we obtain formula \eqref{eq8}. Formula \eqref{eq9} is proved analogously by replacing $z$ with $\frac{1}{z}$. Convergence of the series is guaranteed when $\left|\frac{1}{z}\right| < 1 \Leftrightarrow |z| > 1$, as desired.
\end{proof}

\begin{thm}\label{thm4}
The following relations hold:
\begin{equation}\label{eq10}
\sum_{\sigma = 0}^{\infty} \frac{z^\sigma U_{N+\sigma-1}^{(\sigma)}(z)}{2^\sigma \sigma!} = (1 - z)^{-\frac{N+1}{2}}U_{N-1}\left(\frac{z}{\sqrt{1-z}}\right), \quad |z| < 1;
\end{equation}
\begin{equation}\label{eq11}
\sum_{\sigma = 0}^{\infty} \frac{z^{-\sigma} U_{N+\sigma-1}^{(\sigma)}(z)}{2^\sigma \sigma!} = \left(\frac{z}{z-1}\right)^{\frac{N+1}{2}}U_{N-1}\left(z\sqrt{\frac{z}{1-z}}\right), \quad |z| > 1.
\end{equation}
\end{thm}

\begin{proof}
According to \eqref{eq1}, $\displaystyle U_{N-1}(x) = \sum_{j = 0}^{\left[\frac{N-1}{2}\right]}(-1)^j{N-1-j \choose j}(2x)^{N-1-2j}$. Let us represent
$$
(2z)^{N-1-2j}(1-z)^j = \left(\frac{2z}{\sqrt{1-z}}\right)(\sqrt{1-z})^{N-1}.
$$
Equating $2x = \frac{2z}{\sqrt{1-z}}$, i.e. $x = \frac{z}{\sqrt{1-z}}$, we then have
$$
\sum_{j = 0}^{\left[\frac{N-1}{2}\right]}(-1)^j{N-1-j \choose j}\frac{(2z)^{N-1-2j}}{(1-z)^{N-j}} = (1-z)^{-\frac{N+1}{2}}U_{N-1}\left(\frac{z}{\sqrt{1-z}}\right).
$$
Formula \eqref{eq10} is proved. For formula \eqref{eq11}, we transform the argument as follows:
\begin{eqnarray*}
\lefteqn{(2z)^{N-1-2j}z^Nz^{-j}(z-1)^{-N+j}} \\
&=& (2z)^{N-1-2j}\frac{z^N}{(z-1)^N}\left(\sqrt{\frac{z}{z-1}}\right)^{N-1-2j}\left(\sqrt{\frac{z}{z-1}}\right)^{-N+1}.
\end{eqnarray*}
Now, let $\displaystyle 2x = 2z\sqrt{\frac{z}{z-1}}$ or $\displaystyle x = z\sqrt{\frac{z}{z-1}}$. Then,
\begin{eqnarray*}
\lefteqn{\sum_{j = 0}^{\left[\frac{N-1}{2}\right]}(-1)^j{N-1-j \choose j}\frac{(2)^{N-1-2j}(z)^{2N-1-3j}}{(z-1)^{N-j}}} \\
&=& \frac{z^N}{(z - 1)^N}\left(\sqrt{\frac{z - 1}{z}}\right)^{N-1}U_{N-1}\left(\frac{z\sqrt{z}}{\sqrt{z - 1}}\right),
\end{eqnarray*}
whence \eqref{eq11} follows.
\end{proof}

\subsection{Examples}

Using formula \eqref{eq6}, we have the following:
\begin{itemize}
\item $N = 1 \rightarrow \frac{1}{2^s s!}U_{s}^{(s)}(z) = 1$ and $N = 2 \rightarrow \frac{1}{2^s s!}U_{s + 1}^{(s)}(z) = 2z(s + 1)$;
\item $N = 3 \rightarrow \frac{1}{2^s s!}U_{s+2}^{(s)}(z) = 2z^2(s+1)(s + 2) - (s + 1)$;
\item $N = 4 \rightarrow \frac{1}{2^s s!}U_{s+3}^{(s)}(z) = \frac{4}{3}z^3(s + 1)(s + 2)(s + 3) - 2z(s + 1)(s + 2)$.
\end{itemize}

The results of applying formulas \eqref{eq10} and \eqref{eq11} (or \eqref{eq8} and \eqref{eq9}) are presented in Table~\ref{tab1} and Table~\ref{tab2}.

\begin{table}
\begin{center}
\renewcommand{\arraystretch}{1.5}
\begin{tabular}{|l|c|c|}
\hline
$N$ & Series & Sum \\ \hline
1 & $\sum_{\sigma = 0}^{\infty} z^\sigma$ & $\frac{1}{1 - z}$ \\ \hline
2 & $\sum_{\sigma = 0}^{\infty} z^\sigma(\sigma + 1)$ & $\frac{1}{(1 - z)^2}$ \\ \hline
3 & $\sum_{\sigma = 0}^{\infty} z^\sigma\big(2z^2(\sigma + 1)(\sigma + 2) - (\sigma + 1)\big)$ & $\frac{4z^2 + z - 1}{(1 - z)^3}$ \\ \hline
4 & $\sum_{\sigma = 0}^{\infty} z^\sigma\left(\frac{4}{3}z^3(\sigma + 1)(\sigma + 2)(\sigma + 3) - 2z(\sigma + 1)(\sigma + 2)\right)$ & $\frac{4z(2z^2 + z - 1)}{(1 - z)^4}$ \\ \hline
\end{tabular}
\end{center}
\caption{Values of the series sums for small $N$ (for $|z| < 1$).}\label{tab1}
\end{table}

\begin{table}
\begin{center}
\renewcommand{\arraystretch}{1.5}
\begin{tabular}{|l|c|c|}
\hline
$N$ & Series & Sum \\ \hline
1 & $\sum_{\sigma = 0}^{\infty} z^{-\sigma}$ & $\frac{z}{z - 1}$ \\ \hline
2 & $\sum_{\sigma = 0}^{\infty} z^{-\sigma}(\sigma + 1)$ & $\frac{z^2}{(z - 1)^2}$ \\ \hline
3 & $\sum_{\sigma = 0}^{\infty} z^{-\sigma}\big(2z^2(\sigma + 1)(\sigma + 2) - (\sigma + 1)\big)$ & $\frac{z^2(4z^3 - z + 1)}{(z - 1)^3}$ \\ \hline
4 & $\sum_{\sigma = 0}^{\infty} z^{-\sigma}\left(\frac{4}{3}z^3(\sigma + 1)(\sigma + 2)(\sigma + 3) - 2z(\sigma + 1)(\sigma + 2)\right)$ & $\frac{4z^4(2z^3 - z + 1)}{(z - 1)^4}$ \\ \hline
\end{tabular}
\end{center}
\caption{Values of the series sums for small $N$ (for $|z| > 1$).}\label{tab2}
\end{table}

Next, let us substitute $z = \frac{i}{2}$ into \eqref{eq10} and use formula \eqref{eq3}.

\begin{cor}\label{cor1}
The following relation holds:
$$
\sum_{\sigma = 0}^{\infty}\left(\frac{i}{2}\right)^\sigma F_N^{(\sigma)} = (-i)^{N - 1}\left(1 - \frac{i}{2}\right)^{-\frac{N + 1}{2}}U_{N-1}\left(\sqrt{-\frac{2 + i}{10}}\right).
$$
\end{cor}

\subsection{Examples, $F_N^{(s)}$}

Now, let us use explicit expressions for the numbers $F_N^{(s)}$ in terms of $s$:
\begin{itemize}
\item $N = 1 \rightarrow\sum_{\sigma = 0}^{\infty} \left(\frac{i}{2}\right)^\sigma = \frac{4}{5}\left(1 + \frac{i}{2}\right)$;
\item $N = 2 \rightarrow \sum_{\sigma = 0}^{\infty} \left(\frac{i}{2}\right)^{\sigma}(\sigma + 1) = \frac{4}{5^2}(3 + 4i)$;
\item $N = 3 \rightarrow \sum_{\sigma = 0}^{\infty} \left(\frac{i}{2}\right)^\sigma \frac{(\sigma + 1)(\sigma + 4)}{2} = \frac{4}{5^3}(19 + 42i)$;
\item $N = 4 \rightarrow \sum_{\sigma = 0}^{\infty} \left(\frac{i}{2}\right)^\sigma \frac{(\sigma + 1)(\sigma + 2)(\sigma + 9)}{6} = \frac{4^2}{5^4}(3 + 79i)$;
\item $N = 5 \rightarrow \sum_{\sigma = 0}^{\infty} \left(\frac{i}{2}\right)^\sigma \frac{(\sigma + 1)(\sigma + 2)(\sigma + 4)(\sigma + 15)}{24} = \frac{4^3}{5^5}\left(-39 + \frac{1159}{8}i\right)$.
\end{itemize}

Now, let us substitute into \eqref{eq11}
$$
z = \begin{cases}
\frac{i}{2}L_k, & k \mbox{ is odd} \\ & \\
\frac{1}{2}L_k, & k \mbox{ is even}
\end{cases}
$$
and make use of formula \eqref{eq4}.

\begin{cor}\label{cor2}
The following relations hold:
$$
\sum_{\sigma = 0}^{\infty} \left(\frac{2}{L_k}\right)^\sigma\Phi_{N,k}^{(\sigma)} = \left(\frac{L_k}{L_k - 2}\right)^{\frac{N+1}{2}}U_{N - 1}\left(\frac{L_k}{2}\sqrt{\frac{L_k}{L_k - 2}}\right), \ k \mbox{ is even};
$$
\begin{align*}
\sum_{\sigma = 0}^{\infty} \left(\frac{-2i}{L_k}\right)^\sigma\Phi_{N,k}^{(\sigma)} &= (-i)^{N-1}\left(\frac{L_k}{L_k + 2i}\right)^{\frac{N+1}{2}}U_{N - 1}\left(i\frac{L_k}{2}\sqrt{\frac{L_k}{L_k + 2i}}\right), \\
& \qquad k \geq 3 \mbox{ is odd}.
\end{align*}
\end{cor}

Let $z_0 = \frac{1}{2}e^{i\arctan(\sqrt{15})} = \frac{1}{8}(1 + i\sqrt{15})$. It is easily shown that $\frac{z_0}{\sqrt{1 - z_0}} = \frac{i}{2}$. Using formula \eqref{eq10} then yields the following corollary.

\begin{cor}\label{cor3}
The following relation holds:
$$
(-i)^{N-1}(1 - z_0)^{\frac{N+1}{2}}\sum_{\sigma=0}^\infty\frac{(z_0)^\sigma U_{N+\sigma-1}^{(\sigma)}(z_0)}{2^\sigma \sigma!} = F_N.
$$
\end{cor}

Expression \eqref{eq10} defines a function analytic in the disc $|z| < 1$. However, on the boundary of the disc the values of this function are defined everywhere except at the point $z = 1$. Consider the analytic continuation of this function at the point $z = e^{it}$.

\begin{cor}\label{cor4}
The following formulas hold:
\begin{equation}\label{eq12}
\sum_{\sigma = 0}^{\infty} e^{it\sigma}\frac{U_{N + \sigma - 1}^{(\sigma)}(e^{it})}{2^\sigma \sigma!} = \left(2\sin\left(\frac{t}{2}\right)\right)^{-\frac{N+1}{2}}e^{-\frac{(t - \pi)(N + 1)}{4}i}U_{N-1}\left(\frac{e^{\frac{3t + \pi}{4}i}}{\sqrt{2\sin(t/2)}}\right),
\end{equation}
\begin{equation}\label{eq13}
\sum_{\sigma = 0}^{\infty} e^{-it\sigma}\frac{U_{N + \sigma - 1}^{(\sigma)}(e^{it})}{2^\sigma \sigma!} = \left(2\sin\left(\frac{t}{2}\right)\right)^{-\frac{N+1}{2}}e^{\frac{(t - \pi)(N + 1)}{4}i}U_{N-1}\left(\frac{e^{\frac{5t - \pi}{4}i}}{\sqrt{2\sin(t/2)}}\right),
\end{equation}
where the sums are understood in the sense of the analytic continuation of the corresponding functions.
\end{cor}

Let us consider the particular cases of formula \eqref{eq12}.
\begin{itemize}
\item Let $t = \frac{\pi}{3}$. Then,
$$
\sum_{\sigma = 0}^{\infty}e^{i\pi\sigma/3}\frac{U_{N+\sigma-1}^{(\sigma)}(e^{i\pi/3})}{2^\sigma \sigma!} = e^{\frac{(N+1)\pi}{6}i}U_{N-1}(i),
$$
whence, taking \eqref{eq5} into account, we obtain
$$
e^{-\frac{(2N-1)\pi}{3}i}\sum_{\sigma = 0}^{\infty}e^{i\pi\sigma/3}\frac{U_{N+\sigma-1}^{(\sigma)}(e^{i\pi/3})}{2^\sigma \sigma!} = P_N.
$$

\item Let $t = \frac{\pi}{2}$. Then,
$$
\sum_{\sigma = 0}^{\infty}e^{i\pi\sigma/2}\frac{U_{N+\sigma-1}^{(\sigma)}(e^{i\pi/2})}{2^\sigma \sigma!} = (2)^{-\frac{N+1}{4}}e^{\frac{(N+1)\pi}{8}i}U_{N-1}(2^{-1/4}e^{i5\pi/8}),
$$
whence, taking \eqref{eq5} into account, we obtain
$$
\sum_{\sigma = 0}^{\infty}(i)^\sigma P_N^{(\sigma)} = (2)^{-\frac{N+1}{4}}e^{\frac{(3N-5)\pi}{8}i}U_{N-1}(2^{-1/4}e^{i5\pi/8}).
$$

Examples:

$\displaystyle
N=1: \sum_{\sigma=0}^\infty (i)^\sigma = \frac12+\frac i2;
$

$\displaystyle
N=2: \sum_{\sigma=0}^\infty (i)^\sigma(2+2\sigma) = i, \mbox{whence}
$
$$
\qquad \sum_{\sigma=0}^\infty (i)^\sigma\sigma =- \frac12;
$$

$\displaystyle
N=3: \sum_{\sigma=0}^\infty (i)^\sigma(1+\sigma)(5+2\sigma) =-1+ i\frac32, \mbox{whence}
$

$$\hspace{1cm} \sum_{\sigma=0}^\infty (i)^\sigma\sigma^2 =- \frac i2;
$$

$\displaystyle
N=4: \sum_{\sigma=0}^\infty (i)^\sigma\frac23(1+\sigma)(2+\sigma)(9+2\sigma) =-3+ i,\mbox{whence}
$

$$\hspace{.5cm} \sum_{\sigma=0}^\infty (i)^\sigma\sigma^3 =1,
$$
etc. The above formulas were known to Euler.

\item Letting $t = \pi$ yields the following lemma.
\end{itemize}

\begin{lemma}\label{lem4}
The following equality holds:
$$
\frac{1}{2^s s!}U_{N + s - 1}^{(s)}(-1) = (-1)^{N-1}{N + 2s \choose N - 1}.
$$
\end{lemma}

\begin{proof}
Formula (4) in \cite{ref4} implies that
$$
\frac{1}{2^s s!}U_{N + s - 1}^{(s)}\left(\frac{1}{2}\left(z^{1/2} + z^{-1/2}\right)\right) = z^{-\frac{N - 1}{2}} \sum_{j = 0}^{N-1}{N + s - 1 - j \choose s}{s + j \choose j}z^j.
$$
Hence,
$$
\frac{1}{2^s s!}U_{N + s - 1}^{(s)}(1) = \sum_{j = 0}^{N - 1}{N + s - 1 - j \choose N - 1 - j}{s + j \choose j}.
$$
A variation of Vandermonde's identity \cite{ref11} gives the following,
$$
\sum_{j = 0}^{N - 1}{N + s - 1 - j \choose N - 1 - j}{s + j \choose j} = {N + 2s \choose N - 1}.
$$
This can be illustrated by appealing to a counting example of placing balls into distinguishable bins. Quite famously, ${n + k - 1 \choose k - 1} = {n + k - 1 \choose n}$ counts the number of ways of distributing $n$ balls into $k$ distinguishable bins. Then, let there be $2s + 2$ distinguishable containers and $N - 1$ balls. We have that ${N - 1 + s - j \choose N - 1 - j}$ counts the number of ways to distribute $N - 1 - j$ balls into the first $s + 1$ distinguishable containers, while ${s + j \choose j}$ is the number of ways of distributing $j$ balls into the next $s + 1$ containers. Thus, summing over all $j = 0$ to $N - 1$ just yields the total number of ways of distributing $N - 1$ balls into $2s + 2$ containers. Whence,
$$
\frac{1}{2^s s!}U_{N + s - 1}^{(s)}(1) = {N + 2s \choose N - 1}.
$$

Since $U_{N + s - 1}^{(s)}(z) = (-1)^{N - 1}U_{N + s - 1}^{(s)}(-z)$, the lemma is proved.
\end{proof}

From letting $t = \pi$ in \eqref{eq13} and using Lemma~\ref{lem4}, it follows that
$$
\sum_{\sigma = 0}^{\infty}{(-1)^\sigma}{N + 2\sigma \choose N - 1} = (-1)^{N - 1}2^{-\frac{N + 1}{2}}U_{N - 1}\left(-\frac{1}{\sqrt{2}}\right),
$$
whence
\begin{equation}\label{eq14}
\sum_{\sigma = 0}^{\infty}(-1)^{\sigma}{N + 2\sigma \choose N - 1} = 2^{-\frac{N}{2}} \sin\left(\frac{N\pi}{4}\right).
\end{equation}

Note that
$$
2^{-\frac{N}{2}} \sin\left(\frac{N\pi}{4}\right) = \begin{cases}
0, & N \equiv 0 (\mbox{mod} \ 4) \\
\frac{(-1)^{\left[\frac{N}{4}\right]}}{2^{\left[\frac{N+1}{2}\right]}}, & \mbox{otherwise}
\end{cases}.
$$

Let us consider examples.
\begin{itemize}
\item $N = 1 \rightarrow \sum_{\sigma = 0}^{\infty}(-1)^\sigma = \frac{1}{2}$;
\item $N = 2 \rightarrow \sum_{\sigma = 0}^{\infty}(-1)^\sigma(2 + 2\sigma) = \frac{1}{2}$, whence
$$
\sum_{\sigma = 0}^{\infty}(-1)^{\sigma}\sigma = -\frac{1}{4};
$$
\item $N = 3 \rightarrow \sum_{\sigma = 0}^{\infty}(-1)^\sigma\frac12(2 + 2\sigma)(3 + 2\sigma) = \frac{1}{4}$, whence
$$
\sum_{\sigma = 0}^{\infty}(-1)^{\sigma}\sigma^2 = 0;
$$
\item $N = 4 \rightarrow \sum_{\sigma = 0}^{\infty}(-1)^\sigma\frac{1}{6}(2 + 2\sigma)(3 + 2\sigma)(4 + 2\sigma) = 0$,
whence
$$
\sum_{\sigma = 0}^{\infty}(-1)^\sigma \sigma^3 = \frac{1}{8};
$$
\end{itemize}
and so on. These formulas coincide with Euler's formulas.

More generally, formula \eqref{eq14} can be considered for any complex number $N$ if we replace ${N + 2\sigma \choose N - 1}$ by $\frac{\Gamma(N + 2\sigma + 1)}{\Gamma(N) \Gamma(2 + 2\sigma)}$. Then, for an arbitrary complex $z$ (including the poles of the corresponding Gamma functions), we obtain
\begin{equation}\label{eq15}
\sum_{\sigma = 0}^{\infty}\frac{\Gamma(z + 2\sigma + 1)}{\Gamma(2 + 2\sigma)} = \Gamma(z)2^{-z/2}\sin\left(\frac{z\pi}{4}\right).
\end{equation}
For $|z| \geq 1$ the series can be understood in the sense of the analytic continuation of the corresponding function. Formula \eqref{eq15} generalizes Euler's formulas for summations of alternating power series.

\section{Conclusion}

In the present article, we develop the properties of the specific polynomials $U_{N + s - 1}^{(s)}(z)$, i.e. derivatives of Chebyshev polynomials of the second kind in which the order of the polynomial is related to the order of the derivative. Such polynomials are related to the Gegenbauer polynomials $C_N^{(\alpha)}(z)$ \cite{ref12} by the relation
$$
\frac{1}{2^s s!}U_{N + s - 1}^{(s)}(x) = C_{N - 1}^{(s + 1)}(x).
$$
Formulas \eqref{eq10} and \eqref{eq11} can be rewritten as follows:
$$
\sum_{\sigma = 0}^{\infty}z^\sigma C_{N}^{(1 + \sigma)}(z) = \left(\frac{1}{1-z}\right)^{\frac{N + 2}{2}}C_{N}^{(1)}\left(\frac{z}{\sqrt{1 - z}}\right), \quad |z| < 1;
$$
$$
\sum_{\sigma = 0}^{\infty}z^{-\sigma} C_{N}^{(1 + \sigma)}(z) = \left(\frac{z}{z - 1}\right)^{\frac{N + 2}{2}}C_{N}^{(1)}\left(z\sqrt{\frac{z}{z - 1}}\right), \quad |z| > 1.
$$
Similarly, from Corollaries~\ref{cor3} and \ref{cor4} we obtain formulas connecting a classical combinatorial sequence with certain aspects of the theory of special functions and complex sequences:
$$
F_N = (-i)^{N - 1}(1 - z_0)^{\frac{N + 1}{2}} \sum_{\sigma = 0}^{\infty} (z_0)^\sigma C_{N - 1}^{(1 + \sigma)}(z_0),
$$
where $z_0 = \frac{1}{2}e^{i\arctan(\sqrt{15})} = \frac{1}{8}(1 + i\sqrt{15})$, and
$$
P_N=e^{-i\frac\pi3(2N-1)}\sum_{\sigma=0}^\infty e^{i\pi\sigma/3} C_{N-1}^{(1+\sigma)}(e^{i\pi/3}).
$$

And then there are ``magic'' values of the argument for which the infinite sum of transcendental functions degenerates into a discrete sequence of integers which might contain zeros as particular values, i.e.
$$
\sum_{\sigma = 0}^{\infty}(1/2)^\sigma C_{N}^{(1 + \sigma)}(1/2) = (2)^{\frac{N + 2}{2}} C_N^{(1)}\left(\frac{1}{\sqrt{2}}\right) = (2)^{\frac{N + 3}{2}}\sin\left(\frac{(N + 1)\pi}{4}\right).
$$
With other values of the argument, it is also possible to obtain various integer sequences:  
the sequence of natural numbers is obtained via the relation
$$
N+1=(1-z_0)^{\frac{N+2}2}\;\sum_{\sigma=0}^\infty (z_0)^\sigma C_N^{(1+\sigma)}(z_0),
$$
where $\displaystyle z_0=\frac{\sqrt 5-1}2$, thus $\displaystyle\frac{z_0}{\sqrt{1-z_0}}=1;$
and the Jacobsthal sequence is obtained via the relation
$$
J_{N+1}=(-i)^N2^{\frac N2}(1-z_0)^{\frac{N+2}2}\; \sum_{\sigma=0}^\infty (z_0)^\sigma C_N^{(1+\sigma)}(z_0),
$$
where $\displaystyle z_0=\frac{1+i\sqrt{31}}{16}$, thus $\displaystyle\frac{z_0}{\sqrt{1-z_0}}=\frac i{2\sqrt 2}.$

Note that the series of Gegenbauer polynomials in which the summation is taken over the upper index do not occur in the classical handbooks.

In this work, exact analytic expressions are obtained for the sums of infinite series involving the polynomials $U_{N + s - 1}^{(s)}(z)$ (or $C_{N - 1}^{(s + 1)}(z)$). These sums are expressed in terms of Chebyshev polynomials with complex argument. It is shown that derivatives of Chebyshev polynomials are an effective tool for representing convolved Fibonacci numbers, Pell numbers, and the $k$-sections of the Fibonacci sequence. This made it possible to derive new identities for classical numerical sequences. The possibility of evaluating sums of formally divergent series by analytic continuation is demonstrated, which in particular leads to results consistent with the classical Euler formulas. The applied role of derivatives of Chebyshev polynomials is described, for example, in \cite{ref13}. The obtained results may be applied in approximation theory and in the development of numerical methods using orthogonal polynomials. Moreover, in view of the formula
$$
U_{M - 1}(T_N(x))U_{N - 1}(x) = U_{NM-1}(x),
$$
they may also prove useful in cryptography \cite{ref14}.

\bibliography{references}
\bibliographystyle{amsplain}

\end{document}